\documentclass{amsart}

\usepackage{amssymb}
\usepackage{color}
\usepackage{amsxtra} 
\usepackage{mathrsfs} 
\usepackage{yfonts}

\newtheorem{theorem}{Theorem}
\newtheorem{lem}[theorem]{Lemma}
\newtheorem{prop}[theorem]{Proposition}
 
\newtheorem*{theorem*}{Main Theorem} 
\newtheorem*{theorem**}{Theorem}

\theoremstyle{definition}
\newtheorem{definition}[theorem]{Definition}

\theoremstyle{remark}
\newtheorem{remark}[theorem]{Remark} 
\newtheorem*{remark*}{Remark}

\numberwithin{equation}{section}
\numberwithin{theorem}{section} 

\begin{document}

\title[A
n instability mechanism of  pulsatile  flow 
]
{
An instability mechanism of pulsatile flow
 along
particle trajectories
 for the axisymmetric Euler equations
}


\author{Tsuyoshi Yoneda}
\address{Graduate School of Mathematical Sciences, University of Tokyo, Komaba 3-8-1 Meguro, Tokyo 153-8914, Japan} 
\email{yoneda@ms.u-tokyo.ac.jp}


\date{\today} 



\begin{abstract} 
The dynamics along the particle trajectories for the 3D axisymmetric Euler equations in an infinite cylinder 
 are considered. 
It is shown that 
if the inflow-outflow is rapidly increasing in time, the corresponding  
 laminar profile of the Euler flow is not (in some sense) stable
provided that 
the swirling component  is not small. 
This exhibits an instability mechanism of pulsatile flow. 
In the proof, Frenet-Serret formulas and orthonormal moving frame are essentially used.
\end{abstract} 

\maketitle

\section{Introduction} 
\label{sec:Intro} 
We study the dynamics along  the particle trajectories for the 3D axisymmetric Euler equations.
Such  Lagrangian dynamics have already been studied in mathematics (see \cite{C0, C1,C2}). For example, in \cite{C1}, Chae considered a blow-up problem for the axisymmetric 3D incompressible Euler equations with swirl. 
More precisely, he showed that under some assumption of local minima for the pressure on the axis of symmetry
with respect to the radial variations along some particle trajectory, the solution blows up in finite time.


Although the blowup problem of 3D Euler equations is still an outstanding open problem, in this paper, we focus on a different problem in physics, 
especially,  the cardiovascular system \cite{FQV}. 
If the blood flow is in large and medium sized vessels, the flow is governed by the usual incompressible Navier-Stokes equations. 
In this study field, Womersley number is the key. The Wormersley number
comes from oscillating (in time) solutions to the  Navier-Stokes equations in a tube.
Let us explain more precisely.
We define a pipe $\Omega_{\mathcal R}$ as  $\Omega_{\mathcal R}:=\{x\in\mathbb{R}^3: \sqrt{x_1^2+x_2^2}<\mathcal R,\ 0<x_3<\ell
\}$
with its side-boundary  $\partial\Omega_{\mathcal R}=\{x\in\mathbb{R}^3: \sqrt{x_1^2+x_2^2}=\mathcal R,\ 0<x_3<\ell
\}$.
The incompressible Navier-Stokes equations are described as follows:
\begin{equation}
\label{NS eq.}
\partial_tu+(u\cdot \nabla)u-\nu\Delta u=-\nabla p,\quad\nabla \cdot u=0\quad \text{in}\quad \Omega,
\quad u=0 \quad \text{on}\quad \partial\Omega_{\mathcal R}
\end{equation}
with $u=u(x,t)=(u_1(x_1,x_2,x_3,t),u_2(x_1,x_2,x_3,t),u_3(x_1,x_2,x_3,t))$ and 
 $p=p(x,t)$.

To give the Womersley number, we need to focus on the axisymmetric Navier-Stokes flow without swirl (see \cite{W}).
If $p_1$ and $p_2$ are the pressure at the ends of the pipe $\Omega_{\mathcal R}$,
the 
pressure gradient can be expressed as $(p_1-p_2)/\ell$.
If the pressure gradient is time-independent, $(p_1-p_2)/\ell=:p_s$, 
then we can find the stationary Navier-Stokes flow (Poiseuille flow):
\begin{equation}\label{p_s}
u_s=(u_1,u_2, u_3)=(0, 0, \frac{p_s}{4\nu \ell}(\mathcal R^2-r^2)),
\end{equation}
where $r=\sqrt{x_1^2+x_2^2}$.
Note that $u_s$ is also a solution to the  linearized Navier-Stokes equations.
Next we  consider the oscillating pressure gradient case,
\begin{equation}\label{n}
\frac{p_1(t)-p_2(t)}{\ell}=p_oe^{i N t}
\end{equation}
which is periodic in the time.
Then its corresponding solution $u_o$  can be written explicitly by using a Bessel function (see \cite[(8)]{W} and \cite[(1)]{TKWP}) with $u_1=u_2=0$. Thus $u_o$ is also a solution to the linearized Navier-Stokes equations.
Now we can give the Womersley number $\alpha$  as follows: 
\begin{equation*}
\alpha=\mathcal R\sqrt{\frac{N}{\nu}}.
\end{equation*}
In \cite{TKWP}, they also defined the oscillatory Reynolds number and 
the mean Reynolds number by using $u_o$ and $u_s$ respectively,
and they investigated how the transition of pulsatile flow from the laminar to the 
turbulent (critical Reynolds number) is affected by the Womersley number and the oscillatory Reynolds number.
According to their experiment,
measurement at different Womersley numbers yield similar transition behavior,
and variation of the oscillatory Reynolds number also appear to have little effect.
Thus they conclude that the transition seems to be determined only by the mean Reynolds number.
However it seems they did not investigate  the effect of  the non-small swirl component (azimuthal component),
 and thus
our aim here is to 
 show that the non-small swirl component induces an instability which is, at a glance, nothing to do with wall turbulence.
Let us explain more precisely.
Since we  would not like to  take the boundary layer into account,  it is reasonable to consider  
a simpler model: the 3D axisymmetric Euler flow in an infinite cylinder
$\Omega:=\{x\in\mathbb{R}^3: \sqrt{x_1^2+x_2^2}<1,\ x_3\in\mathbb{R}
\}$ (the setting $\Omega$ is just for simplicity). 
The incompressible Euler equations are expressed as follows:
\begin{eqnarray}
\label{Euler eq.}
& &\partial_tu+(u\cdot \nabla)u=-\nabla p,\quad\nabla \cdot u=0\quad \text{in}\quad \Omega,\\
\nonumber
& & \quad u|_{t=0}=u_0,\quad u\cdot n=0 \quad \text{on}\quad \partial\Omega,\quad
u(x,t)\to (0,0,g(t))\quad (x_3\to\pm\infty)
\end{eqnarray}
with $u=u(x,t)=(u_1(x_1,x_2,x_3,t),u_2(x_1,x_2,x_3,t),u_3(x_1,x_2,x_3,t))$, 
 $p=p(x,t)$ and an uniform (in space) inflow-outflow condition  $g=g(t)$ (the uniform setting is just for simplicity,
we can easily generalize it), where $n$ is a unit normal vector on the boundary.

\begin{remark}
According to the boundary layer theory,
outside the boundary layer the fluid motion is accurately described by the Euler flow.
Thus the above  simplification seems (more or less) valid.
For the recent progress on the mathematical analysis of the boundary layer, 
see \cite{MM}.
\end{remark}
Roughly saying, the inflow-outflow $g$ is a simplification of  $u_s+u_o$,
namely,  $u_s$ is approximated by the mean-value of $g$,
$\partial_t u_o$ and $\partial_t^2 u_o$ are  approximated by $g'$ and $g''$ respectively.
Since we consider the axisymmetric Euler flow, we can simplify the Euler equations \eqref{Euler eq.}.
Let $e_r:= x_h/|x_h|$,
$e_\theta:=x_h^\perp/|x_h|$ and
  $e_z=(0,0,1)$ with
 $x_h=(x_1,x_2,0)$, $x_h^\perp=(-x_2,x_1,0)$.
The vector valued function $u$ can be rewritten as   $u=v_re_r+v_\theta e_\theta+v_ze_z$,  
where $v_r=v_r(r,z,t)$, $v_\theta=v_\theta(r,z,t)$ and $v_z=v_z(r,z,t)$
with $r=|x_h|$ and $z=x_3$.
Then the   axisymmetric Euler equations can be expressed  
as follows: 
\begin{eqnarray}
\partial_t v_r+v_r\partial_rv_r+v_z\partial_zv_r-\frac{v_\theta^2}{r}+\partial_r p&=&
0,\\
\label{axisymmetricEuler-1}
\partial_tv_\theta+v_r\partial_rv_\theta+v_z\partial_zv_\theta+\frac{v_rv_\theta}{r}&=&
0,\\
\partial_tv_z+v_r\partial_rv_z+v_z\partial_zv_z+\partial_zp&=&0
,\\
\label{axisymmetricEuler-2}
\frac{\partial_r(rv_r)}{r}+\partial_zv_z&=&0.
\end{eqnarray}

In order to show that  the non-small swirl component induces the instability,
we need to measure appropriately the rate of laminar profile of the Euler flow.  
\begin{definition}\label{Stream-shell near the boundary} 
(Axis-length  streamline in $z$.)\ 
For a unilateral flow $v_z=u\cdot e_z>0$ in $\Omega$,
 we can define an axis-length streamline $\gamma(z)$.
Let $t$
 be fixed, and let $\gamma(z)$ be such that 
\begin{equation*}
\gamma(\bar r_0,z,t)=\gamma (z):=(\bar R(z)\cos \bar \Theta(z), \bar  R(z)\sin \bar \Theta(z), z)
\end{equation*} 
with
$\bar R(z)=\bar R(\bar r_0,z,t)$, $\bar R(\bar r_0,-\infty,t)=\bar r_0$, $\bar \Theta(z)=\bar \Theta(z,t)$
 and we choose $\bar R$ and $\bar \Theta$ in order to satisfy 
\begin{equation*}
\partial_z\gamma(z)=\left(\frac{u}{u\cdot e_z}\right)(\gamma(z),t).
\end{equation*}
\end{definition}
We easily see 
\begin{equation*}
\partial_z\gamma\cdot e_z=1,\quad \partial_z\gamma\cdot e_r=\partial_z\bar R=\frac{v_r}{v_z}\quad\text{and}\quad
\partial_z \gamma\cdot e_\theta=\bar R\partial_z\bar\Theta=\frac{v_\theta}{v_z}.
\end{equation*}
In this paper we always assume existence of  a unique smooth solution to the Euler equations. 
Since $\partial_{\bar r_0}\bar R>0$ due to the smoothness,
we have 
its inverse $r_0=\bar R^{-1}(r,z,t)$.
We now give the key definition.
\begin{definition} (Rate of laminar profile.)\ 
Let $\partial=\partial_{z}$ or $\partial_{\bar r_0}$, and let $\bar \partial=\partial_{z}$ or $\partial_r$. We 
define ``rate of laminar profile" $L^x$ and $L^t$ as follows:
\begin{equation*}
L^x(\bar r_0,z,t)
:=
\sum_{\ell=1}^3
|\partial^\ell\bar R(\bar r_0,z,t)|
+
 \sum_{\ell=1}^3|(\bar \partial^\ell\bar R^{-1})(\bar R(\bar r_0,z,t),z,t)|
\end{equation*}
and
\begin{equation*}
L^t(\bar r_0,z,t)=
|(\partial_t\bar R^{-1})(\bar R(\bar r_0,z,t),z,t)|
+|(\partial_t\partial_{\bar r_0}\bar R)(\bar r_0,z,t)|+|(\partial_t\partial_z\bar R)(\bar r_0,z,t)|.
\end{equation*}
Later we deal with the curvature and torsion of the particle trajectory, thus it is natural to see up to three derivatives.
\begin{remark}
As we already assumed that  solutions to the Euler equations are always unique and smooth enough,
thus, when the transition of  the Euler flow from the laminar to the turbulent regime
occurs, then $L^x$ and/or $L^t$ must tend to infinity.
\end{remark}

\begin{remark}
Minumum value of $L^x$ is $2$, since $|\partial_r\bar R^{-1}|=1/|\partial_{\bar r_0}\bar R|$.
\end{remark}

\end{definition}

\begin{remark}\label{no swirl}
We easily see that $u=(0,0,g)$
 is one of the solution to \eqref{Euler eq.}.
This flow is the typical laminar flow.
In this case 
\begin{equation*}
L^x\equiv 2
\quad\text{and}\quad
L^t\equiv 0.
\end{equation*}
\end{remark}
Now we give the main theorem.

\begin{theorem}

Assume there is a unique  smooth solution to the Euler equations \eqref{Euler eq.}
with smooth initial data satisfying $|L^x|\leq \beta$ for some positive constant $\beta$ (we will determine $\beta$ later).
For any $x\in \Omega$ satisfying $u_0(x)\cdot e_\theta\approx 1$
and $x\cdot e_r>1/\beta$, 
 and any  $\epsilon>0$, then there is  $\delta>0$
such that 
\begin{equation}\label{only occur}
L^t(\bar r_0,z,0)
\gtrsim 1/\epsilon
\end{equation}
with any smooth inflow-outflow $g(t)$ satisfying
\begin{equation}\label{flux condition}
1/\beta^5\leq g(0)\leq 1/\epsilon^5\quad\text{and}\quad 
\frac{1}{\delta^3}<\frac{g'(0)}{\delta^2}<g''(0),
\end{equation}
where $(\bar r_0, z)=\gamma^{-1}(x,0)$.
\end{theorem}


\begin{remark}
In contrast with \cite{TKWP},  
the transition of the pulsatile flow from the laminar to the turbulent may be 
affected by the Womersley number and the oscillatory Reynolds number
provided by the non-small swirl component.
\end{remark}



\begin{remark}
This instability mechanism is (at a glance) different from wall turbulence.
It would be interesting to consider the interaction between this instability mechanism
and wall turbulence, and this is our future work.
\end{remark}

In the next section, we prove the main theorem.

\section{Proof of the main theorem.}

Notations ``$\approx$" and ``$\lesssim$" are convenient. The notation ``$a \approx b$" means there is a positive constant $C>0$ such that 
\begin{equation*}
C^{-1}a\leq b\leq Ca,
\end{equation*}
and ``$a\lesssim 1$" means that there is a positive constant $C>0$ such that 
\begin{equation*}
0\leq a\leq C.
\end{equation*}
This constant $C$ is not  depending on neither $\epsilon$ nor $\delta$. 
Throughout this paper we use $C(\beta)$ (different from the above $C$) as a positive constant depending on $\beta$.
Now we define the particle trajectory.
The associated Lagrangian flow $\eta(t)$
is a solution of the initial value problem 
\begin{align} \label{eq:flowC} 
&\frac{d}{dt}\eta(x,t) = u(\eta(x,t),t), 
\\  \label{eq:flow-icC} 
&\eta(x,0) = x. 
\end{align} 
To prove the main theorem, it is enough to show the following lemma:
\begin{lem}
Assume there is a unique  smooth solution to the Euler equations \eqref{Euler eq.}
with smooth initial data satisfying $|L^x|\leq \beta$ for some positive constant $\beta$.
For any $x\in \Omega$ satisfying $u_0(x)\cdot e_\theta\approx 1$ and $x\cdot e_r>1/\beta$, 
 and any  $\epsilon>0$, then there is  $\delta>0$
such that 
for any small time interval $I$ with $|I|< \beta^2\epsilon^5$,
at least either of the following four cases must happen:
\begin{itemize}

\item
$L^x(\bar r_0,z,t)>1/\beta$,

\item  
$L^t(\bar r_0,z,t)\gtrsim
1/\epsilon$, 

\item  
  $|\eta(x,t)\cdot e_r|<\beta$,

\item

$\bar r_0<\beta$,

\end{itemize}
for some $t\in I$, with any inflow-outflow $g(t)$ satisfying
\begin{equation*}
1/\beta^5\leq g(t)\leq 1/\epsilon^5\quad\text{and}\quad 
\frac{1}{\delta^3}<\frac{g'(t)}{\delta^2}<g''(t)\quad\text{in}\quad
t\in I,
\end{equation*}
where $(\bar r_0, z)=(\gamma^{-1}\circ\eta)(x,t)$
(in this case $\bar r_0$ and $z$ are depending on $t$).
\end{lem}
Since the time interval $I$ is arbitrary, we see that $L^x$ or $\eta\cdot e_r$ or $\bar r_0$ is not continuous at the initial time $t=0$,
or $L^t\gtrsim 1/\epsilon$ at the initial time.
The discontinuity contradicts the smoothness assumption, 
thus 
\begin{equation*}
L^t\gtrsim 
1/\epsilon
\end{equation*}
only occurs.
In what follows, we prove the above lemma.
For any time interval $I$ with $|I|\leq \beta^2\epsilon^5$,
assume that the axisymmetric smooth Euler flow satisfies the following both conditions:
\begin{itemize}
\item
$L^x(\bar r_0,z,t)\leq 1/\beta$ 
and
$
L^t(\bar r_0,z,t)
\lesssim 1/\epsilon$ for
any $t\in I$,

\item

$|\eta(x,t)\cdot e_r|\geq \beta$ 
and 
$\bar r_0\geq \beta$
for any $t\in I$,
\end{itemize}
where 
 $(\bar r_0,z)=(\gamma^{-1}\circ\eta)(x,t)$.
First we express $v_z$ and $v_r$ by using $\bar R$ and $\bar R^{-1}$. To do so, we define the cross section of the stream-tube (annulus). Let 
$B_{-\infty}(\bar r_0)=\{x\in\mathbb{R}^3: |x_h|<\bar r_0,\ x_3=-\infty\}$ and let
\begin{equation*}
  A(\bar r_0,z,\epsilon,t):=\bigcup_{x\in B_{-\infty}(\bar r_0+\epsilon)\setminus B_{-\infty}(\bar r_0)}
\gamma(x,z,t).
\end{equation*}
We see that  its measure is  
\begin{equation*}
|A(\bar r_0,z,\epsilon,t)|=\pi\left(\bar R(\bar r_0,\epsilon,z,t)^2-\bar R(\bar r_0,z,t)^2\right).
\end{equation*}
\begin{definition} (Inflow propagation.)\ 
Let $\rho$ be such that  

\begin{equation*}
\rho(\bar r_0,z,t):=\lim_{\epsilon\to 0}\frac{|A(\bar r_0,-\infty,\epsilon,t)|}{|A(\bar r_0,z,\epsilon,t)|}.
\end{equation*}
\end{definition}
We see that
\begin{equation*}
\rho(\bar r_0,z,t)=\frac{\partial_{\bar r_0}\bar R(\bar r_0,-\infty,t) \bar R(\bar r_0,-\infty,t)}{\partial_{\bar r_0}\bar R(\bar r_0,z,t) \bar R(\bar r_0,z,t)}=\frac{\bar r_0}{\partial_{\bar r_0}\bar R(\bar r_0,z,t) \bar R(\bar r_0,z,t)}
=\frac{2\bar r_0}{\partial_{\bar r_0}\bar R(\bar r_0,z,t)^2}.
\end{equation*}
\begin{remark}
By the assumption on $L^x$, 
we have  the 
estimates of the inflow propagation $\rho$:
\begin{equation*}
|\partial_z\rho|, |\partial_z^2\rho|,
|\partial_{\bar r_0}\rho|, |\partial_{\bar r_0}^2\rho|
\lesssim C(\beta)
\quad\text{for}\quad t\in I
\end{equation*}
with some positive constant $C(\beta)$.
Note that $\bar R(\bar r_0, z,t)>\beta$ for $(\bar r_0,z)=(\gamma^{-1}\circ \eta)(x,t)$ due to $\eta(x,t)\cdot e_r>\beta$.
\end{remark}

Since 
\begin{equation*}
2\pi \int_{\bar R(\bar r_0,z,t)}^{\bar R(\bar r_0+\epsilon,z,t)}u_z(r',z,t)r'dr'=2\pi\int_{\bar r_0}^{\bar r_0+\epsilon}u_z(r',-\infty,t)r'dr'
\end{equation*}
 by divergence-free and Gauss's divergence theorem,
 we can figure out $v_z$ by using the inflow propagation $\rho$, 
\begin{eqnarray*}
v_z(r,z,t)&=&\lim_{\epsilon\to 0}\frac{2\pi}{|A(\bar r_0,z,\epsilon,t)|}\int_{\bar R(\bar r_0,z,t)}^{\bar R(\bar r_0+\epsilon,z,t)}v_z(r',z,t)r'dr'\\
&=&
\lim_{\epsilon\to 0}\frac{|A(\bar r_0,-\infty,\epsilon,t)|}{|A(\bar r_0,z,\epsilon,t)|}\frac{2\pi}{|A(\bar r_0,-\infty,\epsilon,t)|}\int_{\bar r_0}^{\bar r_0+\epsilon}v_z(r',-\infty,t)r'dr'\\
&=&
\rho(\bar r_0,z,t)u_z(\bar r_0,-\infty,t).
\end{eqnarray*}
Thus we have the following proposition.

\begin{prop}\label{formula of velocities}
We have the following formula of $v_z$ and $v_r$:
\begin{equation}\label{!}
v_z(r,z,t)=\rho(\bar R^{-1}(r,z,t),z,t)v_z(\bar R^{-1}(r,z,t),-\infty,t)=\rho(\bar R^{-1},z,t)g(t)
\end{equation}
and 
\begin{equation}\label{!!}
v_r(r,z,t)=(\partial_z\bar R)(\bar R^{-1}(r,z,t), z,t)v_z(r,z,t).
\end{equation}
\end{prop}
By the above proposition, we see 
\begin{equation*}
\beta^{-6}\lesssim|v_z|\lesssim \beta^{-2}\epsilon^{-5}\quad\text{and}\quad |v_r|\lesssim \beta^{-3}\epsilon^{-5}.
\end{equation*}

We now define the Lagrangian flow along $r$,$z$-direction. Let 
\begin{eqnarray}\label{2D-trajectory-1}
& &\frac{d}{dt}Z(t)=v_z(R(t),Z(t),t),\\
\nonumber
& &Z(0)=z_0
\end{eqnarray}
and
\begin{eqnarray}\label{2D-trajectory-2}
& &\frac{d}{dt}R(t)=v_r(R(t),Z(t),t),\\
\nonumber
& &R(0)=r_0
\end{eqnarray}
with $Z(t)=Z(r_0,z_0,t)$ and $R(t)=R(r_0,z_0,t)$.
By the second assumption: $|\eta(x,t)\cdot e_r|\geq \beta$, $R$ satisfies the following: 
\begin{equation*}
R(t)>\beta
\quad\text{for}\quad t\in I.
\end{equation*}

Since $v_z>0$, then we can define the inverse of $Z$ in $t$: $t=Z_t^{-1}(z,r_0,z_0)$.
In this case we can estimate $\partial_zZ^{-1}_t=1/\partial_tZ=1/v_z
\lesssim \beta^6$ and $\partial_z^2Z^{-1}_t=-\frac{\partial_zv_z}{v_z^2}$.
First we show the following estimates.
\begin{lem}\label{estimates of v}
For $t\in I$,
we have the following estimates along the axis-length trajectory:
\begin{equation}\label{partial_zv_z}
\begin{cases}
\partial_zv_z(R(Z^{-1}_t(z)),z,Z_t^{-1}(z))\approx (\partial_tZ)^{-1}\rho g'+\text{remainder},\\
\partial_z^2Z_t^{-1}\approx- (\partial_tZ)^{-3}\rho g'+\text{remainder},\\
\partial_z^2v_z(R(Z^{-1}_t(z)),z,Z_t^{-1}(z))\approx  (\partial_tZ)^{-2}\rho g''+\text{remainder},\\
|\partial_zv_r(R(Z^{-1}_t(z)),z,Z_t^{-1}(z))|\lesssim  (\partial_tZ)^{-1}\rho g'+\text{remainder},\\
|\partial_z^2v_r(R(Z^{-1}_t(z)),z,Z_t^{-1}(z))|\lesssim  (\partial_tZ)^{-2}\rho g''+\text{remainder},
\end{cases}
\end{equation}
where ``remainder"  is small compare with the corresponding main term provided by small $\delta>0$.
Moreover, we have
\begin{equation}\label{v_theta estimate}
v_\theta(R(Z^{-1}_t(z)),z,Z^{-1}_t(z))\approx 1
\end{equation} 
for some $z$ sufficiently close to $z_0$,
\begin{eqnarray}\label{partial_zv_theta}
& &|\partial_zv_\theta(R(Z^{-1}_t(z)),z,Z_t^{-1}(z))|\lesssim C(\beta),\\
\nonumber
& &|\partial_z^2v_\theta(R(Z^{-1}_t(z)),z,Z_t^{-1}(z))|\lesssim (\partial_t Z)^{-3}\rho g'+\text{remainder}
\end{eqnarray}
and
\begin{equation*}
\partial_t|u(\eta(x,t),t)|\lesssim C(\beta)g'+\text{remainder}.
\end{equation*}
\end{lem}

\begin{proof}
The estimates of $v_r$ and $v_z$ with several derivatives are just direct calculation
using the formulas \eqref{!} and \eqref{!!}. 
The point is just extracting the main terms composed by $g'$ or $g''$.
Here we  control $v_\theta$ by using \eqref{partial_zv_z}.
By \eqref{axisymmetricEuler-1} we see that 
\begin{equation*}
\partial_t v_\theta(R(t),Z(t),t)=v_\theta(r_0,z_0,0)-\frac{v_r(R(t),Z(t),t)v_\theta(R(t),Z(t),t)}{R(t)}.
\end{equation*}
Applying the Gronwall equality,  we see
\begin{equation}\label{v_theta}
v_\theta(R(Z^{-1}_t(z)),z,Z^{-1}_t(z))=v_\theta(r_0,z_0,0)\exp\bigg\{-\int_0^{Z_t^{-1}(z)}
\frac{v_r(R(\tau),Z(\tau),\tau)}{R(\tau)}d\tau\bigg\}.
\end{equation}

\begin{remark}\label{NS difficult}
To obtain the formula of $v_r$ and $v_z$, we are only using the function $\bar R$, $g$
and Gauss's divergence theorem. Namely, we do not need to use  the Euler equations.
However, to obtain the above formula of $v_\theta$, we essentially use the Euler equation \eqref{axisymmetricEuler-1}.
Thus, to pursue the  the Navier-Stokes flow case,
 we need some new idea.
\end{remark}

Since $|v_r|\lesssim \beta^{-3}\epsilon^{-5}$, $R(t)>\beta$ and choose $z$ sufficiently close to $z_0$, 
 we have \eqref{v_theta estimate}.
More precisely, we choose a point $Z(t)$ such that  
\begin{equation*}
|Z(t)-z_0|\lesssim \beta^{-4}\epsilon^5.
\end{equation*}
In this case, by $Z(t)=z_0+\int_0^tv_z(R(\tau),Z(\tau),\tau)d\tau$ and $|v_z|\gtrsim\beta^{-6}$,
the time interval $I$ always satisfies $|I|\lesssim \beta^2\epsilon^5$.
Just taking derivatives to \eqref{v_theta} in $z$-valuable, then we also have \eqref{partial_zv_theta}.
Now we estimate $\partial_t|u(\eta,t)|$.
We set  the usual trajectory $\eta(x,t)$ using 
smooth functions $R$, $Z$ and $\Theta$:
\begin{equation*}
\eta(x,t)
=(R(t)\cos\Theta(t),R(t)\sin\Theta(t),Z(t))
\end{equation*}
with
$e_\theta=(-\sin\Theta(t),\cos\Theta(t),0)$ and 
$e_r=(\cos\Theta(t),\sin\Theta(t),0)$. Then, by a direct calculation with $u=v_re_r+v_\theta e_\theta+v_ze_z$,  we see that 
\begin{equation}\label{multiply}
\frac{1}{2}\partial_t|u(\eta(x,t),t)|^2=\partial_t u\cdot u
=\partial_t v_rv_r+\partial_tv_\theta v_\theta+\partial_tv_zv_z
\end{equation}
along the trajectory.
In fact, since
\begin{equation*}
\partial_t\eta=\partial_t R(\cos\Theta,\sin\Theta,Z)+\partial_t\Theta(-R\sin\Theta,R\cos\Theta,Z),
\end{equation*}
and 
\begin{equation*}
v_\theta=\partial_t\eta\cdot e_\theta=\partial_t\Theta R,
\end{equation*}
we see $\partial_t\Theta=v_\theta/R$.
We multiply $u=v_r e_r+v_\theta e_\theta+v_ze_z$ to
\begin{equation*}
\partial_t u=\partial_tv_r e_r+\partial_tv_\theta e_\theta+\partial_tv_z e_z+
v_r\partial_t\Theta e_\theta-v_\theta\partial_t\Theta e_r,
\end{equation*}
then we have \eqref{multiply}.
Just take a time derivative to $v_z$ along the trajectory, then we have  
\begin{equation*}
\partial_t(v_z(R(t),Z(t),t))=
\rho g'+\text{remainder}.\\
\end{equation*}
Here we used the fact that 
$
L^t
\lesssim 1/\epsilon$.
Thus
\begin{equation*}
\partial_t v_z v_z=\rho^2g'g+\text{remainder}\quad\text{for}\quad t\in I.
\end{equation*}
The remainder becomes small compare with the main term provided by small $\delta>0$.
By the similar calculation,
\begin{equation*}
\partial_t v_r v_r
=
(\partial_z\bar R)^2\rho^2g'g
+\text{remainder}
\quad\text{for}\quad t\in I.
\end{equation*}
Here we also used the fact that $
L^t
\lesssim 1/\epsilon$.
Clearly
$|\partial_tv_\theta|\lesssim \beta^{-4}\epsilon^{-5}$ and then
 $|\partial_tv_\theta v_\theta|\lesssim \beta^{-4}\epsilon^{-5}$.
Thus we have 
\begin{equation*}
\partial_t|u(\eta(x,t),t)|^2\approx (1+(\partial_z\bar R)^2)\rho^2g'g
+\text{remainder}\quad\text{for}\quad t\in I
\end{equation*}
and then
\begin{equation*}
\partial_t|u(\eta(x,t),t)|\approx \rho g'
+\text{remainder}
\quad\text{for}\quad t\in I.
\end{equation*}
\end{proof}

Our strategy  is to estimate the curvature and torsion of the arc-length particle trajectory.
To do so, we need to define the axis-length trajectory $\tilde\eta$ in $z$.
\begin{definition} (Axis-length trajectory.)\ 
Let $\tilde\eta$ be such that 
\begin{equation*}
\tilde \eta(z):=(r(z)\cos\theta(z),r(z)\sin\theta(z),z)
\end{equation*}
and we choose $r(z)$ and $\theta(z)$
in order to satisfy $\tilde \eta(z)=\eta(x,Z^{-1}_t(z))$.
\end{definition}
For $t\in I$, we see
\begin{eqnarray*}
\partial_z\tilde\eta\cdot e_\theta&=&\frac{\partial_t\eta\cdot e_\theta}
{v_z}=r\theta'=\frac{v_\theta(R(Z^{-1}_t(z)),z,Z^{-1}_t(z))}{v_z(R(Z^{-1}_t(z)),z,Z^{-1}_t(z))}
\lesssim \beta^6,\\
\partial_z\tilde\eta\cdot e_r&=&\frac{v_r}{v_z}=(\partial_z\bar R)(\bar R^{-1},Z,t)=r',\quad |r'|\lesssim C(\beta,\epsilon)\quad\text{and}\quad |r''|\lesssim C(\beta,\epsilon)
\end{eqnarray*}
with some positive constant $C(\beta,\epsilon)$ depending on $\beta$ and $\epsilon$.
In particular we need the estimates of $\theta''$ and $\theta'''$.
These estimates specify the curvature and torsion of the particle trajectory.
By Lemma \ref{estimates of v}, we can immediately obtain the following proposition.
\begin{prop}\label{estimates of theta}
By Lemma \ref{estimates of v}
 we have (just see the highest order term)
\begin{equation*}
\theta''(z)= -\frac{v_\theta\partial_zv_z}{v_z^2}+\text{remainder}=  -C(\beta)g'+\text{remainder}
\end{equation*}
 and 
\begin{eqnarray*}
\theta'''(z)&=& -\frac{v_\theta\partial_z^2v_z}{rv_z^2}+\frac{2v_\theta(\partial_zv_z)^2}{rv_z^3}+\text{remainder}\\
&=&
 -C(\beta)g''+\text{remainder}\quad \text{for}\quad  t\in I, 
\end{eqnarray*}
where ``remainder" is small compare with the corresponding main term provided by small $\delta>0$.
\end{prop}

From  the trajectory $\eta(x,t)$, we define the arc-length trajectory $\eta^*(s)=\eta^*(x,s)$.

\begin{definition} (Arc-length trajectory.)\ 
Let $\eta^*$ be such that 
\begin{equation*}
\eta^*(s)=\eta^*(x,s):=\eta(x,t(s))\quad \text{and}\quad \eta^*(x,0)=\eta(x,0)
\end{equation*}
with  $\partial_st(s)=|u|^{-1}$.
\end{definition}
In this case we see $|\partial_s\eta^*(s)|=1$.
We define the unit tangent vector $\tau$ as 
\begin{equation*}
\tau(s)=\partial_s \eta^*(x,s),
\end{equation*}
the unit curvature vector $n$ as $\kappa n=\partial_s \tau$ with a curvature function $\kappa(s)>0$,
the unit torsion vector $b$ 
as : $b(s):=\pm\tau(s)\times n(s)$ ($\times$ is an exterior product)
 with a torsion function to be positive $T(s)>0$ (once we restrict $T$ to be positive, then the direction of $b$ can be uniquely determined), that is,
\begin{equation*}
Tb:=\partial_sn+\kappa \tau,\quad |b|=1
\end{equation*}
due to the Frenet-Serret formula.
By the estimates of $\theta''$ and $\theta'''$ in  Proposition \ref{estimates of theta},
 we  obtain the following key estimates.
\begin{lem}\label{key estimate on torsion}
For any $\epsilon>0$,
we have 
\begin{equation*}
(1/6)|u|^2|\partial_s\kappa|> \kappa\partial_t|u|,
(1/2)|\partial_s\kappa|>  |\kappa Tb\cdot e_\theta|
\quad\text{and}\quad
-1\leq (n\cdot e_\theta)<-1/2
\end{equation*}
for sufficiently small  $\delta>0$. 
\end{lem}

\begin{proof}
Recall  the arc-length  trajectory ($z=z(s)$): 
\begin{equation*}
\eta^* (x,s)=\tilde\eta(x,z)=(r(z)\cos \theta(z), r(z)\sin \theta(z), z)\quad\text{with}\quad \theta'>0.
\end{equation*}
Thus
$\tau$ and  $\kappa n$ are expressed as 
\begin{equation*}
\tau=(\partial_z\tilde\eta)z',\quad \kappa n=\partial_s^2\eta^*=\partial_z^2\tilde\eta(z')^2+\partial_z\tilde\eta z''.
\end{equation*}
Also we immediately have the following formula:
\begin{equation*}
\partial_s^3\eta^*=\partial_s(\kappa n)=\partial_z^3\tilde \eta(z')^3+3\partial_z^2\tilde\eta z'z''+\partial_z\tilde\eta z'''.
\end{equation*} 
We recall that 
\begin{equation*}
\partial_z\tilde\eta\cdot e_z=1,\quad \partial_z\tilde\eta\cdot e_\theta=r\theta'=\frac{u_\theta}{u_z}
\lesssim \beta^6
\end{equation*}
and then we see that 
\begin{equation*}
\theta'\lesssim \beta^5\quad\text{and}\quad
\theta''\approx-C(\beta)g'.
\end{equation*}
Also recall that 
\begin{equation*}
\partial_z\tilde\eta\cdot e_r=r'=\frac{u_r}{u_z}=(\partial_z\bar R)(\bar R^{-1},Z,t).
\end{equation*}
Direct calculations yield (just see the highest order term composed by $\theta''$ or $\theta'''$,
and neglect the small terms composed by $\theta'$) 
\begin{eqnarray*}
\partial_z\tilde \eta(x, z)&=&
(-r\theta'\sin\theta,r\theta'\cos\theta,1)+(r'\cos\theta,r'\sin\theta, 0)
,\\
\partial_z^2\tilde \eta(x,z)&=&
-r(\theta')^2(\cos\theta,\sin\theta,0)+(-r\theta''\sin\theta, r\theta''\cos\theta,0)
\\
& &
+
r''(\cos\theta,\sin\theta,0)+2r'\theta'(-\sin\theta,\cos\theta,0)\\
&=&
r\theta''(-\sin\theta,\cos\theta,0)
+\text{remainder},
\\
\partial_z^3\tilde \eta(x,z)
&\approx&
r\theta'''(-\sin\theta,\cos\theta,0)+\text{remainder},\\
z'(s)&=&|\partial_z\tilde\eta|^{-1}=(1+(r')^2+(r\theta')^2)^{-1/2}= (1+(r')^2)^{-1/2}+\text{remainder},\\
z''(s)&=&-(1+(r')^2+(r\theta')^2)^{-2}(r'r''+r\theta'(r'\theta'+r\theta''))\\
&=&
-(1+(r')^2)^{-2}r^2\theta'\theta''+\text{remainder},\\
z'''(s)
&\approx&
C(\beta)\theta'''+\text{remainder}.
\end{eqnarray*}
Therefore
\begin{eqnarray*}
\partial_s^2\eta^*\cdot e_\theta&=&\kappa n\cdot e_\theta\\
&=&
r\theta''(1+(r')^2)^{-1}
-r\theta'(1+(r')^2)^{-2}r^2\theta'\theta''
+\text{remainder},
\end{eqnarray*}
\begin{eqnarray*}
\kappa^2&=&|\kappa n|^2
=|\partial_z^2\tilde \eta|^2(z')^4+
2(\partial_z\tilde\eta\cdot\partial_z^2\tilde\eta)(z')^2z''+
|\partial_z\tilde\eta|^2(z'')^2\\
&=&
(r\theta'')^2(1+(r')^2)^{-2}
+((r')^2+1)(1+(r')^2)^{-4}(r^2\theta'\theta'')^2+
\text{remainder}
\end{eqnarray*}
and 
\begin{eqnarray*}
n\cdot e_\theta&=&\frac{\kappa n\cdot e_\theta}{\kappa}\\
&=&
-\frac{r(1+(r')^2)^{-1}-(1+(r')^2)^{-2}r^2\theta'}{\left(r^2(1+(r')^2)^{-2}+(1+(r')^2)^{-3}r^4(\theta')^2\right)^{1/2}}+\text{remainder}.
\end{eqnarray*}
The remainders are small provided by small $\delta$.
We choose $\beta$ (which is independent of $\epsilon$ and $\delta$) such that the main term in $n\cdot e_\theta$ to be strictly smaller than $-1/2$
(by taking small $\beta$, then $\theta'$ becomes small).
Thus we have 
\begin{equation}\label{cancellation}
n\cdot e_\theta=\frac{\kappa n\cdot e_\theta}{\kappa}<-1/2.
\end{equation}
Also recall that $\theta'''\approx -C(\beta)g''$.
The dominant term of $\partial_s(\kappa^2)$  is composed by $\theta''$ and $\theta'''$, more precisely, 
\begin{equation*}
\partial_s(\kappa^2)= 2(\partial_s\kappa) \kappa= 
2r\theta''(r\theta''')(1+(r')^2)^{-5/2}+\text{remainder}\approx
C(\beta)\theta''\theta'''. 
\end{equation*}

We also see that 
\begin{eqnarray*}
& &\kappa= |r\theta''|(1+(r')^2)^{-1}+\text{remainder},\\
&  &\kappa n\cdot e_\theta= r\theta''(1+(r')^2)^{-1}+\text{remainder},\\
& &\partial_s(\kappa n)\cdot e_\theta=\partial_s^3\tilde\eta\cdot e_\theta
=
r\theta'''(1+(r')^2)^{-3/2}+\text{remainder},\\
& &\partial_s\kappa=\frac{r\theta''(r\theta''')(1+(r')^2)^{-5/2}}{\kappa}+\text{remainder}\\
& &
\ \ \ \ \  
=
-r\theta'''(1+(r')^2)^{-3/2}+\text{remainder}
\end{eqnarray*}
in $t\in I$.
``remainder" is small compare with the main terms provided by small $\delta>0$.
We immediately obtain
$(1/6)|u|^2|\partial_s\kappa|> \kappa\partial_t|u|$
for sufficiently small $\delta>0$,
since the left hand side is $\delta^{-3}$-order, while, the right hand side is $\delta^{-2}$-order.
In order to show $(1/2)|\partial_s\kappa|> |\kappa Tb\cdot e_\theta|$,
we use \eqref{cancellation}. 
By the Frenet-Serret formula,
\begin{equation*}
Tb=\partial_s n+\kappa\tau,
\end{equation*}
we see that 
\begin{equation*}
\kappa Tb\cdot e_\theta=\partial_s(\kappa n)\cdot e_\theta
-(\partial_s\kappa) n\cdot e_\theta+\kappa^2 \tau\cdot e_\theta.
\end{equation*}
Thus, by the direct calculation with \eqref{cancellation}, we can find a cancellation on the  highest order term composed by $\theta'''$, 
 and then we have  
\begin{equation*}
|\kappa Tb\cdot e_\theta|<(1/2) |\partial_s\kappa|
\end{equation*}
for sufficiently  small $\delta>0$.
\end{proof}
In what follows, we use  a differential geometric idea. See Chan-Czubak-Y \cite[Section 2.5]{CCY}, more originally, see Ma-Wang \cite[(3.7)]{MW}.
They considered 2D separation phenomena using fundamental differential geometry. 
The key idea here is  ``local pressure estimate" on a normal coordinate in $\bar \theta$, $\bar r$ and $\bar z$ valuables.
Two derivatives to the scalar function $p$ on the normal coordinate is commutative, namely,
$\partial_{\bar r}\partial_{\bar \theta}p(\bar \theta, \bar r,\bar z)-\partial_{\bar \theta}\partial_{\bar r}p(\bar \theta,\bar r,\bar z)=0$ (Lie bracket).
This fundamental observation is the key to extract the local effect of the pressure. 

\begin{remark}
It should be noticed that Enciso and Peralta-Salas \cite{EP}
considered the existence of Beltrami fields $u$
with a nonconstant proportionality factor $f$:
\begin{equation}\label{Beltrami field}
\nabla \times u=fu,\quad \nabla\cdot u=0\quad\text{in}\quad \mathbb{R}^3.
\end{equation} 
It is well known that a Beltrami field is also a solution of the steady Euler equation in $\mathbb{R}^3$.
They showed that for a generic function $f$, the only vector field $u$ satisfying
\eqref{Beltrami field} is the trivial one $u\equiv 0$.
See (2.12), (3.4) and  (3.6) in \cite{EP} for the specific condition on $f$.
Note that $g_{ij}$  (induced metric of the level set of $f$) is the fundamental component of the condition.
It would be also interesting to consider whether we can apply their method to our unsteady flow problem, and compare with our method. 
\end{remark}

For any point $x\in\mathbb{R}^3$ near the arc-length trajectory $\eta^*$ is uniquely  expressed as $x=\eta^*(\bar \theta)+\bar r n(\bar\theta)+\bar z b(\bar\theta)$ with $(\bar\theta,\bar r,\bar z)\in\mathbb{R}^3$ (the meaning of the parameters $s$ and $\bar \theta$ are the same along the arc-length trajectory).
Thus we have that 
\begin{eqnarray*}
\partial_{\bar \theta}x&=&
\tau+\bar r(Tb-\kappa \tau)+\bar z \kappa n,\\
\partial_{\bar r}x&=& n,\\
\partial_{\bar z}x&=&b.
\end{eqnarray*}
This means that 
\begin{equation*}
\begin{pmatrix}
\partial_{\bar \theta}\\
\partial_{\bar  r}\\
\partial_{\bar z}
\end{pmatrix}
=
\begin{pmatrix}
1-\kappa \bar r& \bar z\kappa& \bar r T\\
0& 1& 0\\
0& 0& 1
\end{pmatrix}
\begin{pmatrix}
\tau\\
n\\
b\\
\end{pmatrix}.
\end{equation*}
\begin{remark}
For any smooth scalar function $f$,
we have 
\begin{equation*}
\partial_{\bar \theta}f(x)=\nabla f\cdot \partial_{\bar \theta}x.
\end{equation*}
$\nabla f$ itself is essentially independent of any coordinates, thus we can regard a partial derivative as a vector.
\end{remark}
By the fundamental calculation, we have  the following inverse matrix:
\begin{equation*}
\begin{pmatrix}
\tau\\
n\\
b\\
\end{pmatrix}
=
\begin{pmatrix}
(1-\kappa \bar r)^{-1}& -\bar zT (1-\kappa \bar r)^{-1}& -\bar r T(1-\kappa \bar r)^{-1}\\
0& 1& 0\\
0& 0& 1
\end{pmatrix}
\begin{pmatrix}
\partial_{\bar \theta}\\
\partial_{\bar  r}\\
\partial_{\bar z}
\end{pmatrix}.
\end{equation*}
Therefore we have the following orthonormal moving frame:
 $\partial_{\bar r}=n$, $\partial_{\bar z}=b$ and 
$$
(1-\kappa \bar r)^{-1}\partial_{\bar \theta}-\bar z T(1-\kappa \bar r)^{-1}\partial_{\bar r}-\bar rT(1-\kappa \bar r)^{-1}\partial_{\bar z}=\tau.
$$

In order to abbreviate the complicated indexes, we re-define the absolute value of the velocity along the trajectory. 
 Let (the indexes are $x$ and $t$ respectively)
\begin{equation*}
|u|:=|u(\eta(x',t),t)|\quad\text{with}\quad
x'=\eta^{-1}(x,t)
\end{equation*}
and 
\begin{equation*}
\partial_t|u|:=\partial_{t'}|u(\eta(x',t'),t')|\bigg|_{t'=t}\quad\text{with}\quad
x'=\eta^{-1}(x,t).
\end{equation*}

\begin{lem}\label{Euler flow along the trajectory}
We see  $-\nabla p\cdot \tau=\partial_t|u|$ along the trajectory.
\end{lem}
\begin{proof}
Let us define a unit tangent vector $\tilde\tau$ (in time $t'$) as follows:
\begin{equation*}
\tilde \tau_{x,t}(t'):=\frac{u}{|u|}(\eta(x',t'),t')\quad\text{with}\quad x'=\eta^{-1}(x,t).
\end{equation*}
Note that  there is a re-parametrize factor $s(t')$ such that 
\begin{equation*}
\tau(s(t'))=\tilde \tau(t').
\end{equation*}
Since $u\cdot \partial_s\tau=0$, we see that 
\begin{eqnarray*}
\partial_{t'}|u(\eta(x',t'),t')|&=&
\partial_{t'}(u(\eta(x',t'),t')\cdot \tilde\tau_{x,t}(t'))\\
&=&
\partial_{t'}(u(\eta(x',t'),t'))\cdot \tilde\tau_{x,t}(t')+
u(\eta (x',t'),t')\cdot\partial_s\tau\partial_{t'}s\\
&=&\partial_{t'}(u(\eta(x',t'),t'))\cdot \tilde\tau_{x,t}(t').
\end{eqnarray*}
By the above calculation we have 
\begin{equation*}
-\nabla p\cdot\tau=\partial_{t'}(u(\eta(x',t'),t')\cdot\tau|_{t'=t}=\partial_{t'}(u(\eta(x',t'),t')\cdot\tilde\tau|_{t'=t}=
\partial_{t'}|u|\bigg|_{t'=t}.
\end{equation*}
\end{proof}

\begin{lem}
Along the arc-length trajectory, we have 
\begin{equation*}
3\kappa\partial_t|u|+\partial_s\kappa|u|^2=\partial_{\bar r}\partial_t|u|
\end{equation*}
and 
\begin{equation*}
T\kappa|u|^2=\partial_{\bar z}\partial_t|u|.
\end{equation*}
\end{lem}

\begin{proof}
By using the orthonormal moving frame, 
we have the following gradient of the pressure,
\begin{equation*}
\nabla p= (\partial_\tau p)\tau+(\partial_np)n+(\partial_bp)b.
\end{equation*}
Recall that
\begin{equation*}
\partial_st=|u|^{-1}
\quad\text{and}\quad
\partial_t\eta\cdot \tau=|u|.
\end{equation*}
By the unit normal vector with the curvature constant, we see
\begin{equation*}
\kappa n=\partial_s^2\eta^*=\partial_s(\partial_t\eta\partial_st)
=\partial_t^2\eta(\partial_st)^2+\partial_t\eta\partial_s^2t.
\end{equation*}
Thus we have 
\begin{eqnarray*}
-(\nabla p\cdot n)&=&
(\partial_t^2\eta\cdot n)=
\kappa|u|^2,\\
-\partial_s(\nabla p\cdot n)&=&\partial_s(\kappa(\partial_st)^{-2})
=\partial_s\kappa(\partial_s t)^{-2}-2\kappa(\partial_st)^{-3}(\partial_s^2t),\\
-\nabla p\cdot \tau&=&-|u|^3\partial_s^2 t,\\
-\nabla p\cdot b&=&0.
\end{eqnarray*}
By Lemma \ref{Euler flow along the trajectory} with
$0=\kappa n\cdot \tau=\partial_t^2\eta\cdot \tau(\partial_s t)^2+\partial_t\eta\cdot \tau(\partial_s^2 t)$,
we also have 
\begin{equation*}
\partial_s^2t=-|u|^{-3}\partial_t|u|.
\end{equation*}
Recall that  
\begin{equation*}
\partial_\tau=(1-\kappa\bar r)^{-1}\partial_{\bar\theta}-\bar zT(1-\kappa\bar r)^{-1}\partial_{\bar r}
-\bar rT(1-\kappa\bar r)^{-1}\partial_{\bar z}.
\end{equation*}
Along the arc-length trajectory, we have 
\begin{eqnarray*}
-\partial_{\bar r}(\nabla p\cdot \tau)
&=&-\partial_{\bar r}\partial_\tau p
\\
&=&
-\kappa\partial_{\bar \theta} p-\partial_{\bar r}\partial_{\bar \theta} p-T\partial_{\bar z} p\\
\nonumber
(\text{commute}\ \partial_{\bar r}\ \text{and}\  \partial_{\bar \theta})
&=&
-\kappa(\nabla p\cdot \tau)-\partial_{\bar\theta} (\nabla p\cdot n)-T(\nabla p\cdot b)\\
&=&
-\kappa |u|^3\partial_s^2t+\partial_s\kappa(\partial_st)^{-2}-2\kappa(\partial_st)^{-3}(\partial_s^2 t)\\
&=&
3\kappa\partial_t|u|+\partial_s\kappa|u|^2.
\end{eqnarray*}
Since $\nabla p\cdot b=\partial_{\bar z} p\equiv 0$ along the trajectory, then
\begin{equation*}
-\partial_{\bar z}(\nabla p\cdot \tau)|_{\bar r,\bar z=0}=
-\partial_{\bar z}\partial_\tau p|_{\bar r,\bar z=0}=
-\partial_{\bar z}\partial_{\bar\theta} p-T\partial_{\bar r} p
=
-T(\nabla p\cdot n)=T\kappa|u|^2.
\end{equation*}

By Lemma \ref{Euler flow along the trajectory}
along the arc-length trajectory $\eta^*$, we have 
\begin{equation*}
3\kappa\partial_t|u|+\partial_s\kappa|u|^2=
-\partial_{\bar r}(\nabla p\cdot \tau)|_{\bar r,\bar z=0}
=
\partial_{\bar r}\partial_t|u|\\
\end{equation*}
and 
\begin{equation*}
T\kappa|u|^2=
-\partial_{\bar z}(\nabla p\cdot \tau)|_{\bar r,\bar z=0}
=
\partial_{\bar z}\partial_t|u|.\\
\end{equation*}

\end{proof}

By using the above lemma we can finally prove the main theorem.
Since 
\begin{equation*}
\partial_\theta=(e_\theta\cdot n)\partial_{\bar r}+(e_\theta\cdot b)\partial_{\bar z}
\end{equation*}
and the axisymmetric flow is rotation invariant (along the $e_\theta$-direction),
\begin{eqnarray*}\label{rotation invariant}
0&=&\partial_\theta\partial_t|u|=(e_\theta\cdot n)\partial_{\bar r}\partial_t|u|
+(e_\theta\cdot b)\partial_{\bar z}\partial_t|u|\\
\nonumber
&=&
3(e_\theta\cdot n)\left(\kappa\partial_t|u|+\partial_s\kappa|u|^2\right)
+(e_\theta\cdot b)T\kappa|u|^2.
\end{eqnarray*}
However, by
Lemma \ref{key estimate on torsion},
we have 
\begin{eqnarray*}
& &\left|3(e_\theta\cdot n)\left(\kappa\partial_t|u|+\partial_s\kappa|u|^2\right)
+(e_\theta\cdot b)T\kappa|u|^2\right|\\
&\geq&
\frac{3|\partial_s\kappa||u|^2}{2}-3\kappa \big|\partial_t|u|\big|-T\kappa |u|^2\\
&\geq&
\frac{|\partial_s\kappa||u|^2}{2}>0,
\end{eqnarray*} 
and it is in contradiction.

\vspace{0.5cm}
\noindent
{\bf Acknowledgments.}\ 
The author would like to thank Professor Norikazu Saito for letting me know 
the book \cite{FQV},
and Professor Hiroshi Suito for letting me know 
``Womersley number".
The author was partially supported by JST CREST.

\bibliographystyle{amsplain}

\end{document}